\documentclass[a4paper]{article}

\usepackage[utf8]{inputenc} 
\usepackage[T1]{fontenc} 
\usepackage{amsmath,amsthm,latexsym} 
\usepackage[fixamsmath]{mathtools}
\mathtoolsset{showmanualtags,mathic,centercolon} 
\usepackage{amssymb,amsfonts} 
\usepackage{booktabs} 
\usepackage{enumitem}
\usepackage[UKenglish]{babel} 
\usepackage[svgnames]{xcolor} 
\usepackage{algpseudocode} 
\usepackage{algorithm} 
\usepackage{bm} 
\usepackage{url}
\usepackage{enumitem}

\usepackage[margin=1.25in]{geometry} 
\usepackage[labelfont=sl,textfont=sl]{subcaption} 
\usepackage[font={small,sl},labelsep=period]{caption} 

\numberwithin{equation}{section} 
\theoremstyle{plain}
\newtheorem{theorem}{Theorem}[section] 
\newtheorem{lemma}[theorem]{Lemma} 
\newtheorem{corollary}[theorem]{Corollary}

\theoremstyle{definition}
\newtheorem{definition}[theorem]{Definition}

\newtheorem{assumption}[theorem]{Assumption}

\theoremstyle{remark}
\newtheorem{remark}[theorem]{Remark}

\newcommand{\thmref}[1]{Theorem~\ref{#1}}

\newcommand{\assref}[1]{Assumption~\ref{#1}}

\newcommand{\remref}[1]{Remark~\ref{#1}}

\newcommand{\R}{\mathbb{R}} 
\newcommand{\bigO}{\mathcal{O}} 
\newcommand{\grad}{\nabla} 

\renewcommand{\b}[1]{\bm{#1}} 

\newcommand{\revision}[1]{#1}

\newcommand{\bx}{\b{x}}

\newcommand{\bs}{\b{s}}
\newcommand{\bv}{\b{v}}

\newcommand{\bg}{\b{g}}

\newcommand{\kappaef}{\kappa_{\textnormal{ef}}}
\newcommand{\kappaeg}{\kappa_{\textnormal{eg}}}
\newcommand{\kappaeh}{\kappa_{\textnormal{eh}}}

\newcommand{\Prob}[1]{\mathbb{P}\left[#1\right]} 

\algrenewcommand\algorithmicrequire{\textbf{Input:}}
\algrenewcommand\algorithmicensure{\textbf{Output:}}

\begin{document}
\title{A note on the complexity of random subspace model-based methods for derivative-free optimization}
\author{
Coralia Cartis\thanks{Mathematical Institute, University of Oxford, Radcliffe Observatory Quarter, Woodstock Road, Oxford, OX2 6GG, United Kingdom (\texttt{cartis@maths.ox.ac.uk}).}
\and 
Lindon Roberts\thanks{School of Mathematics and Statistics \& ARC Training Centre in Optimisation Technologies, Integrated Methodologies, and Applications (OPTIMA), University of Melbourne, Parkville VIC 3010, Australia. ORCID 0000-0001-6438-9703 (\texttt{lindon.roberts@unimelb.edu.au}). 
This author was supported by the Australian Research Council Discovery Early Career Award DE240100006.}}

\date{\today}
\maketitle

\begin{abstract}
	We demonstrate that, with a suitable rescaling, using Johnson--Lindenstrauss transforms (JLTs) in the random subspace model-based derivative-free optimization (DFO) algorithm from [Cartis \& Roberts, \emph{Math.~Prog.} 199 (2023)] achieves an improved \revision{first-order} worst-case evaluation complexity bound compared to `unscaled' JLTs.
    This improved bound does not require any modification to the original algorithm or complexity analysis, and matches the best-known bound in terms of dimension dependence for model-based DFO (e.g.~achieved by [Scheinberg \& Chaudhry, ICM 2026] using Haar matrices in a similar framework).
    \revision{A similarly refined analysis of the analogous second-order subspace method from [Cartis \& Roberts, \emph{Opt.~Meth. \& Softw.} (2026)] shows that using suitably scaled JLTs to generate the subspaces gives a better complexity bound, which improves over full-space methods only if the problem has low effective rank.} 
\end{abstract}

\textbf{Keywords:} derivative-free optimization, large-scale optimization, randomized subspaces, worst case complexity.
\\

\textbf{Mathematics Subject Classification:} 65K05, 90C30, 90C56 

\section{Introduction}
In our previous work \cite{Cartis2023}, we introduced a model-based derivative-free optimization (MBDFO) algorithm that takes steps in randomly generated low-dimensional subspaces at each iteration.
At each iteration, once a subspace has been selected, a quadratic local model for the objective is formed by interpolation to well-chosen points inside the subspace, and this model is minimized (again inside the subspace) to select a new iterate, which is accepted or rejected following a trust-region approach.
In \cite{Cartis2023}, we provided a (high probability) \revision{first-order} worst-case complexity analysis of this algorithm, and, provided the random subspaces are chosen appropriately, showed improved bounds---in terms of explicit dimension dependence---over full-space methods.
This work requires the subspace be constructed using Johnson--Lindenstrauss transforms (JLTs) such as random Gaussian matrices.

In recent work, Scheinberg \& Chaudhry \cite{Scheinberg2026} provided a detailed \revision{first-order} complexity analysis of MBDFO methods, with a particular focus on dimension dependence.
For full-space methods applied to an $n$-dimensional problem, they showed an evaluation complexity of $\bigO(n^2 \epsilon^{-2})$ to achieve first-order stationarity level $\epsilon$ for methods based on full resampling of interpolation points at every iteration and a more practical geometry-correcting algorithm \cite{Scheinberg2010}.
If full resampling is based on a dimension-dependent radius, they show that this can be improved to $\bigO(n\epsilon^{-2})$ evaluations, matching the best-known MBDFO complexity bounds, e.g.~\cite{Grapiglia2023}, although this algorithm is more sensitive to errors (e.g.~rounding errors) in the objective value.

Scheinberg \& Chaudhry also consider a random subspace method based on the Haar ensemble (i.e.~random matrices with orthonormal columns) with full resampling of interpolation points, and show an improved $\bigO(pn\epsilon^{-2})$ evaluation complexity (in expectation), where $p=\bigO(1)$ is the subspace dimension.
This $\bigO(n \epsilon^{-2})$ dependency matches random subspace direct search DFO methods \cite{Roberts2023}, where deterministic direct search methods cannot achieve complexity better than $\bigO(n^2)$ \cite{Dodangeh2016}.

Our results from \cite{Cartis2023}, which are also based on full resampling of interpolation points, proved a worse evaluation complexity of $\bigO(pn^2 \epsilon^{-2})$ (with high probability), where $p=\bigO(1)$ for subspaces generated by JLTs.
\revision{We show that} our previous results can recover the same improved $\bigO(pn\epsilon^{-2})$ evaluation complexity  with no change to the algorithm or complexity analysis; instead, only a simple rescaling of the JLT matrix is required.
\revision{This resolves the question raised in \cite[Remark 3]{Cartis2023}.}

\revision{We then apply the same rescaling idea to the analogous second-order random subspace MBDFO method from \cite{Cartis2026}. We show that applying suitable rescaling achieves a second-order worst-case evaluation complexity of $\bigO(p^3 n^3 \epsilon^{-3})$ (with high probability)---an improvement over the result from \cite{Cartis2026}, namely $\bigO(p^{9/2} n^{9/2} \epsilon^{-3})$ evaluations---compared to $\bigO(n^5 \epsilon^{-3})$ evaluations for standard MBDFO methods. However, in this case we must take $p \sim \epsilon^{-1} r \log r$, where $r \leq n$ is the rank of the (true) objective Hessian.
Hence, for second-order methods, the random subspace only provides an improvement over full-space methods for problems with low effective rank ($r \ll n$).} 

\section{\revision{First-Order Complexity Analysis}}
The MBDFO algorithm from \cite{Cartis2023} uses trust-region methods in the style of \cite{Conn2009,Roberts2026} to solve the unconstrained problem
\begin{align}
    \min_{\bx\in\R^n} f(\bx), \label{eq_prob}
\end{align}
where $f:\R^n\to\R$ is nonconvex with Lipschitz continuous gradient, although the algorithm does not assume access to $\grad f$.
Instead, at each iteration we draw a random $Q_k\in\R^{n\times p}$ which defines a $p$-dimensional subspace (for $p\ll n$) for our next tentative trust-region step.
We build a low-dimensional quadratic model
\begin{align}
    f(\bx_k + Q_k \hat{\bs}) \approx \hat{m}_k(\hat{\bs}) := f(\bx_k) + \hat{\bg}_k^T \hat{\bs} + \frac{1}{2} \hat{\bs}^T \hat{H}_k \hat{\bs}, \label{eq_subspace_model}
\end{align}
via interpolation to points close to $\bx_k$, such that the model is $Q_k$-fully linear\footnote{The full algorithm in \cite{Cartis2023} allows some iterations where the model is not $Q_k$-fully linear, but we make this assumption to align with \cite{Scheinberg2026} and to simplify the exposition.} in the current trust region, $B(\bx_k,\Delta_k)$ \cite[Definition 1]{Cartis2023}:
\begin{align}
    |f(\bx_k+Q_k\hat{\bs}) - \hat{m}_k(\hat{\bs})| \leq \kappaef \Delta_k^2, \qquad \text{and} \qquad \|Q_k^T \grad f(\bx_k+Q_k\hat{\bs}) - \hat{m}_k(\hat{\bs})\| \leq \kappaeg \Delta_k,
\end{align}
for all $\hat{\bs}\in\R^p$ with $\|\hat{\bs}\|\leq \Delta_k$.
The full algorithm \cite[Algorithm 1]{Cartis2023} then follows a standard (trust-region) MBDFO approach. 

For convergence, we require standard assumptions for MBDFO methods---see \cite[Assumptions 1--3]{Cartis2023}---and the following assumption on $Q_k$:

\begin{assumption}[Assumption 4, \cite{Cartis2023}] \label{ass_well_aligned}
    The procedure for generating each $Q_k$ satisfies:
    \begin{enumerate}[label=(\alph*)]
        \item There exists $\alpha_Q>0$ such that, for all $k$, $Q_k$ is \emph{well-aligned} (i.e.~$\|Q_k^T \grad f(\bx_k)\| \geq \alpha_Q \|\grad f(\bx_k)\|$) with probability at least $1-\delta_S$ for some $\delta_S\in(0,1)$, independently of $\{Q_0,\ldots,Q_{k-1}\}$; and
        \item There exists $Q_{\max}>0$ such that $\|Q_k\| \leq Q_{\max}$ for all $k$.
    \end{enumerate}
\end{assumption}

The worst-case complexity of \cite[Algorithm 1]{Cartis2023} is given by the following high probability result.

\begin{theorem} \label{thm_complexity}
    Suppose \cite[Assumptions 1--3]{Cartis2023} and \assref{ass_well_aligned} hold.
    Under suitable assumptions on the trust-region radius updating parameters, if $K_{\epsilon} := \min\{ k : \|\grad f(\bx_k)\| \leq \epsilon\}$ then
    \begin{align}
        \Prob{K_{\epsilon} \leq \bigO(\kappa_H (\kappa_d+1)^2 \alpha_Q^{-2} \epsilon^{-2})} \geq 1 - e^{-\bigO(\kappa_H (\kappa_d+1)^2 \alpha_Q^{-2} \epsilon^{-2})},
    \end{align}
    where $\kappa_d := \max(\kappaef,\kappaeg)$ and $\kappa_H \geq 1$ is a uniform upper bound on the model Hessians, $\|\hat{H}_k\|\leq \kappa_H$.
\end{theorem}
\begin{proof}
    This is identical to \cite[Corollary 1]{Cartis2023}, except with $\kappa_d+1$ in place of $\kappa_d$.
    This more refined estimate of the impact of $\kappa_d$ on the complexity bound follows immediately from the definition of $\epsilon_g(\epsilon)$ and $\Delta^*(\epsilon)$ in \cite[Eqs.~(13) \& (17)]{Cartis2023}.
\end{proof}

In the original result, we provided the complexity in terms of $\bigO(\kappa_d^2)$ rather than $\bigO((\kappa_d+1)^2)$, as we wanted to demonstrate how the constant grows as $\kappa_d$ increases.
However, in \remref{rem_scaling} below, we will consider the case where $\kappa_d$ is small, for which this more refined estimate is relevant.

\subsection{Random Subspace Construction \& Evaluation Complexity}
In \cite[Section 2.6]{Cartis2023}, we show that \assref{ass_well_aligned} can be satisfied with $p=\Omega((1-\alpha_Q)^{-2} |\log\delta_S|)$ (i.e.~with $p=\bigO(1)$ independent of $n$) whenever $Q_k^T$ is a Johnson--Lindenstrauss transform, e.g.~\revision{the Gaussian ensemble}\footnote{\revision{For Gaussian matrices, the bound $Q_{\max}$ holds with high probability; rejecting matrices with large norm give a uniform bound (see \cite[p.~485]{Cartis2023}.}} $[Q_k^T]_{i,j} \sim N(0,1/p)$ i.i.d., or a hashing matrix\footnote{In \cite{Cartis2023} we gave $Q_{\max}=\sqrt{n}$ for hashing JLTs, but \cite[\revision{Lemmas 3.6 \& 3.7}]{Cartis2022c} gives the improved bound $Q_{\max}=\sqrt{n/p}$.}, which both give $Q_{\max}=\bigO(\sqrt{n/p})$.
More examples of suitable JLTs can be found in \cite[Section 3]{Cartis2022c}. 

To get an overall evaluation complexity bound (i.e.~bound the total evaluations of the objective), we note that if the subspace model $\hat{m}_k$ is constructed using linear interpolation to points of distance $\Delta_k$ away from $\bx_k$ (in the subspace), then we get $\kappa_H=1$ and $\kappa_d = \bigO(\sqrt{p}\: Q_{\max}^2)$, where the factor $\sqrt{p}$ comes from standard linear interpolation bounds and the $Q_{\max}^2$ factor comes from the Lipschitz constant of the objective in the subspace, $\hat{\bs} \mapsto f(\bx_k+Q_k\hat{\bs})$.
\revision{The same bounds arise for  interpolation models using diagonal Hessian approximations, with $2p+1$ interpolation points \cite[Corollary 4.16]{Roberts2026}.}
Since the algorithm requires $\bigO(p)$ evaluations per iteration, the iteration bound in \thmref{thm_complexity} implies an evaluation complexity bound of $\bigO(p \kappa_H (\kappa_d+1)^2 \alpha_Q^{-2} \epsilon^{-2})=\bigO(p (\sqrt{p}\: Q_{\max}^2+1)^2 \alpha_Q^{-2} \epsilon^{-2})$.
If $Q_k^T$ is a Gaussian or hashing JLT, per \cite[Section 2.6]{Cartis2023}, for a fixed value of $\alpha_Q$ (and the associated value of $p$), the evaluation complexity is $\bigO(p n^2 \epsilon^{-2})$.

Alternatively, \cite[Lemma 6.7]{Scheinberg2026}, shows\footnote{We note that if $Q_k$ is Haar distributed then the alternative well-aligned definition \cite[Definition 6.2]{Scheinberg2026} is equivalent to our well-aligned condition with $\alpha_Q^2=1-\kappa_g^2$, via \cite[Lemma 6.3]{Scheinberg2026} and noting $\|Q_k \bv\|^2 = \|\bv\|^2$ for all $\bv\in\R^p$.} that taking $Q_k$ to be a Haar matrix (i.e.~uniformly distributed over matrices with orthonormal columns) satisfies \assref{ass_well_aligned} with $\alpha_Q = \bigO(\sqrt{p/n})$, $p=\bigO(1)$ and $Q_{\max}=1$.
Applying these results to Theorem~\ref{thm_complexity} gives the following evaluation complexity bound, which improves on using JLTs by a factor of $n$ (c.f.~\cite[Theorem 6.13]{Scheinberg2026}).

\begin{corollary} \label{cor_haar}
    Suppose the assumptions of Theorem~\ref{thm_complexity} hold.
    If $Q_k$ is generated from the Haar distribution, then \cite[Algorithm 1]{Cartis2023} achieves first-order optimality $\epsilon$ after $\bigO(n\epsilon^{-2})$ iterations or $\bigO(pn\epsilon^{-2})$ objective evaluations, with high probability, where $p=\bigO(1)$ can be chosen independently of $n$.
\end{corollary}

\subsection{Improved Evaluation Complexity for JLTs}
This extra factor of $n$ from the JLT construction can be removed by simply rescaling the JLT, \emph{without changing the main complexity analysis} (i.e.~within the framework of \thmref{thm_complexity}).
This idea arises by noting that our evaluation complexity bound is more sensitive to large $Q_{\max}$ than to small $\alpha_Q$.

Suppose $S\in\R^{p\times n}$ is a Gaussian or hashing JLT.
We can then take our $Q_k$ to be a rescaling of $S$, say $Q_k = \beta S^T$ for some $\beta\in(0,1)$.
This gives us $p=\bigO(1)$ as before, but new values $\alpha_Q=\bigO(\beta)$ and $Q_{\max}=\bigO(\beta \sqrt{n/p})$. 
Applying Theorem~\ref{thm_complexity}, this gives a new evaluation complexity bound of $\bigO(p (\beta^2 n/\sqrt{p}+1)^2 \beta^{-2} \epsilon^{-2})$. 
By taking $\beta=1/\sqrt{n}$, we recover the same evaluation complexity as using Haar matrices (Corollary~\ref{cor_haar}).
We summarize this result below.

\begin{corollary}
    Suppose the assumptions of Theorem~\ref{thm_complexity} hold.
    If $Q_k$ is generated by $Q_k = S^T/\sqrt{n}$, where $S$ is a Gaussian or hashing JLT, then \cite[Algorithm 1]{Cartis2023} achieves first-order optimality $\epsilon$ after $\bigO(n\epsilon^{-2})$ iterations or $\bigO(pn\epsilon^{-2})$ objective evaluations, with high probability, where $p=\bigO(1)$ can be chosen independently of $n$.
\end{corollary}

\begin{remark} \label{rem_scaling}
    We cannot improve the complexity bound by an arbitrary amount by taking an even smaller scaling, $\beta \ll 1/\sqrt{n}$, since the $(\kappa_d+1)^2$ term can only be made as small as $\bigO(1)$ by shrinking $\beta$, which $\beta=1/\sqrt{n}$ achieves.
    The original complexity bound from \cite[Corollary 1]{Cartis2023}, showing the dependency as $\bigO(\kappa_d^2)$, does not make this clear, as the summarized bound was assuming the large $\kappa_d$ limit.
\end{remark}

\begin{remark}
    If the model $\hat{m}_k$ is constructed using gradient information, $\hat{\bg}_k = Q_k^T \grad f(\bx_k)$ (i.e.~we have a derivative-based method) as in \cite{Cartis2022c}, then we have $\kappa_d=\bigO(Q_{\max}^2)$, and the scaled JLT approach gives a bound of $\bigO(n\epsilon^{-2})$ iterations, where each iteration requires an evaluation of a low-dimensional gradient $\hat{\bg}_k\in\R^p$ (e.g.~using $p+1$ evaluations of $f$ for finite differencing).
\end{remark}

\section{\revision{Second-Order Complexity Analysis}}

We now consider the same scaling impact on the second-order random subspace MBDFO method from \cite{Cartis2026}.
In that algorithm, for solving \eqref{eq_prob} in the case where $f$ has Lipschitz continuous Hessians, we again draw a random subspace $Q_k\in\R^{n\times p}$ and build the low-dimensional quadratic model \eqref{eq_subspace_model} via interpolation.
However, we now require the model to be $Q_k$-fully quadratic in $B(\bx_k,\Delta_k)$ \cite[Definition 3.1]{Cartis2026}:
\begin{subequations}
\begin{align}
    |f(\bx_k+Q_k\hat{\bs}) - \hat{m}_k(\hat{\bs})| &\leq \kappaef \Delta_k^3, \\
    \|Q_k^T \grad f(\bx_k+Q_k\hat{\bs}) - \grad \hat{m}_k(\hat{\bs})\| &\leq \kappaeg \Delta_k^2, \\
    \|Q_k^T \grad^2 f(\bx_k + Q_k \hat{\bs}) Q_k - \grad^2 \hat{m}_k(\hat{\bs})\| &\leq \kappaeh \Delta_k,
\end{align}
\end{subequations}
for all $\hat{\bs}\in\R^p$ with $\|\hat{\bs}\| \leq \Delta_k$.
The full algorithm again follows a standard (second-order) MBDFO approach, and is given in \cite[Algorithm 2]{Cartis2026}.

In the second-order setting, the definition of a well-aligned subspace is more complex. 
At iteration $k$, suppose $\grad^2 f(\bx_k)$ has eigendecomposition $\grad^2 f(\bx_k) = \sum_{i=1}^{r} \lambda_i \bv_i \bv_i^T$, where $\lambda_1 \geq \lambda_2 \geq \cdots \geq \lambda_r$ and $r := \operatorname{rank}(\grad^2 f(\bx_k))\leq n$, and $\bv_1,\ldots,\bv_r\in\R^n$ are orthonormal.
For a subspace defined by $P_k\in\R^{n\times p}$, we define $\hat{\bv}_i := P_k^T \bv_i \in \R^p$, and so the (true) subspace Hessian is $P_k^T \grad^2 f(\bx_k) P_k = \sum_{i=1}^{r} \lambda_i \hat{\bv}_i \hat{\bv}_i^T$.

\begin{definition}\label{2:well-aligned}
    The subspace defined by $P_k$ is well-aligned (for second-order convergence) if
    \begin{subequations} \label{eq_well_aligned_2}
    \begin{align}
        \|P_k\| &\leq P_{\max}, \label{eq_well_aligned_2a} \\
        \|P_k^T \grad f(\bx_k)\| &\geq (1-\alpha) \|\grad f(\bx_k)\|, \label{eq_well_aligned_2b} \\
        \|\hat{\bv}_r\| &\geq 1-\alpha, \label{eq_well_aligned_2c} \\
        (\hat{\bv}_i^T \hat{\bv}_r)^2 &\leq 4\alpha^2, \qquad \forall i=1,\ldots,r-1, \label{eq_well_aligned_2d}
    \end{align}
    \end{subequations}
    for some $\alpha\in(0,1)$ and $P_{\max}>0$ independent of $k$.
\end{definition}

We then impose the same probabilistically well-aligned assumption as the first-order case, but again allowing for the well-aligned matrix to be scaled by some constant $\beta>0$.
The below is our scaled alternative to \cite[Assumption 3.7]{Cartis2026}.

\begin{assumption} \label{ass_well_aligned2}
    The procedure for generating each $Q_k$ is such that $Q_k = \beta P_k$ for some $\beta>0$ and $P_k\in\R^{n\times p}$ which is well-aligned (in the sense of Definition~\ref{2:well-aligned}) with probability at least $1-\delta_S$ for some $\delta_S\in(0,1)$, independently of $\{P_0,\ldots,P_{k-1}\}$.
\end{assumption}

Unlike the first-order case, there are some slight changes to the analysis as a result of the rescaling.
We summarize these changes and the new second-order complexity bound below.
To that end, we recall the second-order stationarity measure
\begin{align}
    \sigma_k  := \max(\|\grad f(\bx_k)\|, \tau_k), \qquad\text{where}\qquad \tau_k := \max(-\lambda_{\min}(\grad^2 f(\bx_k)), 0)),
\end{align}
and its corresponding equivalent in the subspace,
\begin{align}
    \hat{\sigma}_k  := \max(\|Q_k^T\grad f(\bx_k)\|, \hat{\tau}_k), \qquad\text{where}\qquad \hat{\tau}_k := \max(-\lambda_{\min}(Q_k^T \grad^2 f(\bx_k) Q_k), 0)),
\end{align}

\begin{lemma} \label{lem_subspace_sigma}
    Suppose \cite[Assumption 3.2]{Cartis2026} holds.
    If $Q_k = \beta P_k$, where $\beta>0$ and $P_k$ is well-aligned (in the sense of Definition~\ref{2:well-aligned}), $\sigma_k \geq \epsilon > 0$, and
    \begin{align}
        \theta := (1-\alpha)^2 - \frac{4M(r-1)\alpha^2}{\epsilon (1-\alpha)^2} > 0, \label{eq_theta_defn}
    \end{align}
    where $M>0$ is an upper bound on $\|\grad^2 f(\bx_k)\|$ from \cite[Assumption 3.2]{Cartis2026}.
    Then
    \begin{align}
        \hat{\sigma}_k \geq \min(\beta (1-\alpha), \beta^2 \theta) \: \epsilon. \label{eq_subspace_sigma}
    \end{align}
\end{lemma}
\begin{proof}
    This is a straightforward modification of \cite[Lemma 3.8]{Cartis2026} to incorporate the scaling $Q_k=\beta P_k$, and noting that $\theta \leq (1-\alpha)^2$.
\end{proof}

The main consequence of Lemma~\ref{lem_subspace_sigma} is that the constants $c_1$ and $c_2$ defined in Lemmas 3.9 \& 3.11 (respectively) of \cite{Cartis2026} are modified to incorporate $\beta$ as it appears in \eqref{eq_subspace_sigma}.
The resulting second-order complexity bound is given below.

\begin{theorem} \label{thm_complexity2}
    Suppose \cite[Assumptions 2.3, 3.2 \& 3.3]{Cartis2026} and \assref{ass_well_aligned2} hold, and $\theta>0$ in \eqref{eq_theta_defn}.
    Under suitable assumptions on the trust-region radius updating parameters, if $K_{\epsilon} := \min\{ k : \sigma_k \leq \epsilon\}$ then
    \begin{align}
        \Prob{K_{\epsilon} \leq \bigO\left(\frac{(\kappa_d+1)^3}{\min(\beta(1-\alpha), \beta^2 \theta)^3} \: \epsilon^{-3}\right)} \geq 1 - \exp\left(-\bigO\left(\frac{(\kappa_d+1)^3}{\min(\beta(1-\alpha), \beta^2 \theta)^3} \: \epsilon^{-3}\right)\right),
    \end{align}
    where $\kappa_d := \max(\kappaef,\kappaeg,\kappaeh,\kappa_H)$ and $\kappa_H \geq 1$ is a uniform upper bound on the model Hessians, $\|\hat{H}_k\|\leq \kappa_H$.
\end{theorem}
\begin{proof}
    Lemma~\ref{lem_subspace_sigma} gives the new constants
    \begin{align}
        c_1 := \frac{\min(\beta (1-\alpha), \beta^2 \theta)}{1+\kappa_{\sigma}\mu^{-1}}, \qquad \text{and} \qquad c_2 := \frac{\min(\beta (1-\alpha), \beta^2 \theta)}{\kappa_{\sigma} + c_0^{-1}},
    \end{align}
    in the statements of \cite[Lemmas 3.9 \& 3.11]{Cartis2026}, where $c_1$ and $c_2$ depend on the algorithm parameter $\mu>0$ and constants $\kappa_{\sigma}$ and $c_0$ defined in \cite[Lemmas 3.4 \& 3.5]{Cartis2026}.
    From \cite[Theorem 3.15]{Cartis2026}, we have $K_{\epsilon} = \bigO\left(\psi(\epsilon)\right)$ with probability at least $1-\exp(-\bigO(\psi(\epsilon)))$ for $\psi(\epsilon)$ defined in \cite[Lemma 3.13]{Cartis2026}.
    By definition, we have
    \begin{align}
        \psi(\epsilon) = \bigO\left(\frac{1}{c_1 \min(\epsilon \Delta_{\min}, \epsilon \Delta_{\min}^2)}\right),
    \end{align}
    with $\Delta_{\min} = \Theta(c_2 \epsilon)$ as $\epsilon\to 0$, and so $\psi(\epsilon) = \bigO(c_1^{-1} c_2^{-2} \epsilon^{-3})$.
    The result then follows from noting $\kappa_{\sigma}=\Theta(\kappa_d)$ and $c_0^{-1} = \Theta(\kappa_d+1)$, which holds in both the large and small $\kappa_d$ limits.
\end{proof}

For standard fully quadratic interpolation error bounds (e.g.~\cite[Corollary 4.8]{Roberts2026}), we may take $\kappaef,\kappaeg,\kappaeh=\bigO(p \beta^3 P_{\max}^3)$, where the $\beta^3 P_{\max}^3$ factor comes from the Lipschitz constant of the subspace-restricted objective Hessian, $\hat{\bs} \mapsto Q_k^T \grad^2 f(\bx_k + Q_k \hat{\bs}) Q_k$ (and noting $\|Q_k\| \leq \beta P_{\max}$ from Assumption~\ref{ass_well_aligned2}), and $\kappa_H=\bigO(\kappaeh + \beta^2 P_{\max}^2)$ from \cite[eq.~(84)]{Cartis2026}, yielding $\kappa_d = \bigO(p \beta^3 P_{\max}^3)$.
This construction requires $\bigO(p^2)$ evaluations per iteration.

\subsection{Subspace Selection}

As described in \cite[Section 3.1]{Cartis2026}, we can construct a well-aligned $P_k$ (in the sense of Definition~\ref{2:well-aligned}) by taking to be a JLT with $p = \Omega(\alpha^{-2} \log r \: |\log \delta_S|)$ and $P_{\max}=\bigO(\sqrt{n/p})$, where $r := \operatorname{rank}(\grad^2 f(\bx_k))$ is the rank of the (true) objective Hessian, $r\leq n$.

However, here we are restricted by the requirement $\theta>0$ \eqref{eq_theta_defn}.
For example, if we want $\theta=(1-\alpha)^2/2$, we require
\begin{align}
    \frac{\alpha^2}{(1-\alpha)^2} = \frac{\epsilon}{8M(r-1)}.
\end{align}
Since the LHS can be approximated by $\alpha^2 + 4\alpha^3 + \bigO(\alpha^4)$ for $\alpha\approx 0$ means we need to take $\alpha^2=\bigO(\epsilon/r)$, and hence our subspace dimension must be of size $p = \Omega(\epsilon^{-1} r \log r)$.

All together, after using $1-\alpha\approx 1$, Theorem~\ref{thm_complexity2} with $P_k$ generated as a JLT gives a (high probability) evaluation complexity of
\begin{align}
    \bigO\left(\frac{p^2 (n^{3/2} p^{-1/2} \beta^3+1)^3}{\min(\beta^3,\beta^6)} \epsilon^{-3}\right) = \bigO\left(\frac{p^2 (n^{9/2} p^{-3/2} \beta^9+1)}{\min(\beta^3,\beta^6)} \epsilon^{-3}\right),
\end{align}
where the equality uses $\bigO((x+1)^3) = \bigO(x^3+1)$ as $x\to 0$ or $x\to\infty$.
Since the numerator is more sensitive to $\beta$ than the denominator (until the ``$+1$'' term in the numerator dominates), we can get an optimal complexity by ensuring  $n^{9/2} p^{-3/2} \beta^9 = \bigO(1)$, i.e.~by choosing $\beta = \bigO(p^{1/6}/\sqrt{n})$.
This gives us the following evaluation complexity result, which improves on the bound for unscaled JLTs from \cite[Section 3.1]{Cartis2026}.

\begin{corollary}
    Suppose the assumptions of Theorem~\ref{thm_complexity2} hold, and $Q_k$ is generated as $Q_k=\beta P_k$ where $\beta=p^{1/6}/\sqrt{n}$ and $P_k$ is a JLT with $p=\Omega(\epsilon^{-1} r \log r)$.
    Then \cite[Algorithm 2]{Cartis2026} achieves second-order optimality $\epsilon$ after $\bigO(p n^3\epsilon^{-3})$ iterations or $\bigO(p^3 n^3 \epsilon^{-3})$ objective evaluations, with high probability.
\end{corollary}

Compared to the complexity of $\bigO(n^5 \epsilon^{-3})$ evaluations for the equivalent full-space (second-order) MBDFO method \cite[Corollary 4.2.1]{Judice2015}, using random subspaces only gives an improved complexity bound when $\epsilon^{-1} r \log r \ll n^{2/3}$ (i.e.~low effective rank problems, $r\ll n$).
As noted in \cite[Remark 3.1]{Cartis2026}, low effective rank problems arise in practice, such as in the finetuning of large language models \cite{Malladi2023} among other examples (see, e.g.~\cite[Section 1]{Cartis-Otemissov2022}).

\section{Conclusion}
Using a JLT ensemble for constructing $Q_k$ can, after scaling by $1/\sqrt{n}$, achieve the same evaluation complexity bound (with high probability) as the Haar ensemble from \cite[Section 6]{Scheinberg2026}, namely $\bigO(pn\epsilon^{-2})$ evaluations to achieve $\epsilon$-first order optimality.
Although this alternative scaling was not considered in \cite{Cartis2023}, that algorithm and analysis (without modification) permits this construction and improved complexity over the `unscaled' JLTs that were originally considered.

\revision{Using scaled JLTs for second-order random subspace methods gives an improved evaluation complexity of $\bigO(p^3 n^3 \epsilon^{-3})$, but $p$ cannot be chosen independently of the problem dimension, and so this improves on full space methods only if the problem has low effective rank.} 

\addcontentsline{toc}{section}{References} 
\bibliographystyle{siam}
\bibliography{refs} 

@article{Cartis-Otemissov2022,
  title = {A Dimensionality Reduction Technique for Unconstrained Global Optimization of Functions with Low Effective Dimensionality},
  author = {Cartis, Coralia and Otemissov, Adilet},
  year = {2022},
  journal = {Information and Inference: A Journal of the IMA},
  volume = {11},
  number = {1},
  pages = {167--201},
  doi = {10.1093/imaiai/iaab011},
}

@misc{Cartis2022c,
  title = {Randomised Subspace Methods for Non-Convex Optimization, with Applications to Nonlinear Least-Squares},
  author = {Cartis, Coralia and Fowkes, Jaroslav and Shao, Zhen},
  year = {2022},
  howpublished = {arXiv preprint arXiv:2211.09873},
  number = {arXiv:2211.09873},
  eprint = {2211.09873},
  primaryclass = {math},
  publisher = {arXiv},
}

@article{Cartis2023,
  title = {Scalable Subspace Methods for Derivative-Free Nonlinear Least-Squares Optimization},
  author = {Cartis, Coralia and Roberts, Lindon},
  year = {2023},
  journal = {Mathematical Programming},
  volume = {199},
  number = {1-2},
  pages = {461--524},
  issn = {0025-5610, 1436-4646},
  doi = {10.1007/s10107-022-01836-1},
}

@article{Cartis2026,
  title = {Randomized Subspace Derivative-Free Optimization with Quadratic Models and Second-Order Convergence},
  author = {Cartis, Coralia and Roberts, Lindon},
  year = {2026},
  journal = {Optimization Methods and Software},
  doi = {10.1080/10556788.2025.2601668},
}

@Book{Conn2009,
  Title                    = {Introduction to Derivative-Free Optimization},
  Author                   = {Conn, Andrew R. and Scheinberg, Katya and Vicente, Lu{\'{i}}s N.},
  Publisher                = {MPS/SIAM},
  Year                     = {2009},

  Address                  = {Philadelphia},
  Series                   = {MPS-SIAM Series on Optimization},
  Volume                   = {8},

  Doi                      = {10.1137/1.9780898718768},
  Url                      = {http://epubs.siam.org/doi/book/10.1137/1.9780898718768}
}

@article{Grapiglia2023,
  title = {Quadratic Regularization Methods with Finite-Difference Gradient Approximations},
  author = {Grapiglia, Geovani Nunes},
  year = {2023},
  journal = {Computational Optimization and Applications},
  volume = {85},
  number = {3},
  pages = {683--703},
  doi = {10.1007/s10589-022-00373-z},
}

@phdthesis{Judice2015,
  title = {Trust-Region Methods Without Using Derivatives: Worst Case Complexity and the Nonsmooth Case},
  author = {J{\'u}dice, D.},
  year = {2015},
  school = {Universidade de Coimbra},
}

@misc{Roberts2026,
  title = {Introduction to Model-Based Derivative-Free Optimization},
  author = {Roberts, Lindon},
  year = {2025},
  howpublished = {arXiv preprint arXiv:2510.04473},
  number = {arXiv:2510.04473},
  eprint = {2510.04473},
  primaryclass = {math.OC},
  publisher = {arXiv},
  doi = {10.48550/arXiv.2510.04473},
}

@article{Roberts2023,
  title = {Direct Search Based on Probabilistic Descent in Reduced Spaces},
  author = {Roberts, Lindon and Royer, Cl{\'e}ment W.},
  year = {2023},
  journal = {SIAM Journal on Optimization},
  volume = {33},
  number = {4},
  pages = {3057--3082},
  doi = {10.1137/22M1488569},
}

@article{Dodangeh2016,
  title = {On the Optimal Order of Worst Case Complexity of Direct Search},
  author = {Dodangeh, M. and Vicente, L. N. and Zhang, Z.},
  year = {2016},
  journal = {Optimization Letters},
  volume = {10},
  number = {4},
  pages = {699--708},
  issn = {1862-4472, 1862-4480},
  doi = {10.1007/s11590-015-0908-1},
}

@inproceedings{Malladi2023,
  title = {Fine-Tuning Language Models with Just Forward Passes},
  booktitle = {37th Conference on Neural Information Processing Systems ({NeurIPS} 2023)},
  author = {Malladi, Sadhika and Gao, Tianyu and Nichani, Eshaan and Damian, Alex and Lee, Jason D and Chen, Danqi and Arora, Sanjeev},
  year = {2023},
}

@Article{Scheinberg2010,
  Title                    = {Self-Correcting Geometry in Model-Based Algorithms for Derivative-Free Unconstrained Optimization},
  Author                   = {Scheinberg, Katya and Toint, Philippe L.},
  Journal                  = {SIAM Journal on Optimization},
  Year                     = {2010},
  Number                   = {6},
  Pages                    = {3512--3532},
  Volume                   = {20},

  Doi                      = {10.1137/090748536},
  Url                      = {http://epubs.siam.org/doi/abs/10.1137/090748536}
}

@inproceedings{Scheinberg2026,
  title = {On Complexity of Model-Based Derivative-Free Methods},
  booktitle = {International Congress of Mathematicians 2026 Volume 7: Invited Lectures: Sections 15--20},
  author = {Scheinberg, K. and Chaudhry, Abraar},
  year = {2026},
  pages = {208--228},
  publisher = {Society for Industrial and Applied Mathematics},
  address = {USA},
  doi = {10.1137/25M1806016},
}

\end{document}